\documentclass[12pt]{amsart}

\usepackage[T1]{fontenc}

\usepackage{lmodern}

\usepackage{hyperref}

\usepackage{amsmath,amssymb,enumerate}
\usepackage[abbrev]{amsrefs}

\usepackage{todonotes}

\usepackage{esint}

\newcommand{\Rn}{\mathbb{R}^n}

\renewcommand{\S}{\mathbb{S}^{n-1}}

\newtheorem{thm}{Theorem}
\newtheorem{lem}[thm]{Lemma}
\newtheorem{prop}[thm]{Proposition}

\newcommand{\dif}{\,\mathrm{d}}

\begin{document}

\title[Borderline Bourgain-Brezis Sobolev embedding theorem]{Variations on a proof of a borderline Bourgain--Brezis Sobolev embedding theorem}
\author{Sagun Chanillo}
\address{Department of Mathematics\\ Rutgers, the State University of New
Jersey\\ 110 Frelinghuysen Road\\ Piscataway, NJ 08854, USA}
\email{chanillo@math.rutgers.edu}
\author{Jean Van Schaftingen}
\address{Institut de Recherche en Math\'ematique et en Physique\\ Universit\'e catholique de Louvain\\ Chemin du Cyclotron 2 bte L7.01.01\\ 1348 Louvain-la-Neuve\\ Belgium}
\email{Jean.VanSchaftingen@uclouvain.be}
\author{Po-Lam Yung}
\address{Department of Mathematics\\ the Chinese University of Hong Kong\\
Shatin\\ Hong Kong}
\email{plyung@math.cuhk.edu.hk}

\dedicatory{To Ha\"im Brezis in admiration and friendship}

\begin{abstract}
We offer a variant of a proof of a borderline Bourgain-Brezis Sobolev embedding theorem on $\mathbb{R}^n$. We use this idea to extend the result to real hyperbolic spaces $\mathbb{H}^n$.
\end{abstract}

\maketitle

\section{Introduction}

The Sobolev embedding theorem states that if $\dot{W}^{1,p}(\mathbb{R}^n)$ is the homogeneous Sobolev space, obtained by completing the set of compactly supported smooth functions $C^{\infty}_c(\mathbb{R}^n)$ under the norm $\|\nabla u\|_{L^p(\mathbb{R}^n)}$, then $\dot{W}^{1,p}(\mathbb{R}^n)$ embeds into $L^{p^*}(\mathbb{R}^n)$, whenever $n \geq 2$, $1 \leq p < n$ and $\frac{1}{p^*} = \frac{1}{p} - \frac{1}{n}$. This fails when $p = n$, i.e.\ $\dot{W}^{1,n}(\mathbb{R}^n)$ does not embed into $L^{\infty}(\mathbb{R}^n)$. One of the well-known remedies of this failure is to say that $\dot{W}^{1,n}(\mathbb{R}^n)$ embeds into $\textrm{BMO}(\mathbb{R}^n)$, the space of functions of bounded mean oscillation. In \citelist{\cite{MR1949165}\cite{MR2293957}}, Bourgain and Brezis established another remedy of the failure of this Sobolev embedding for $\dot{W}^{1,n}(\mathbb{R}^n)$. They proved, among other things, that if $X$ is a differential $\ell$-form on $\mathbb{R}^n$ with $\dot{W}^{1,n}(\mathbb{R}^n)$ coefficients, where $1 \leq \ell \leq n-1$, then there exists a differential $\ell$-form $Y$, whose components are all in $\dot{W}^{1,n} \cap L^{\infty} (\mathbb{R}^n)$, such that
$$
d Y = d X,
$$
with
$$
\|Y\|_{\dot{W}^{1,n} \cap L^{\infty}} \leq C \|dX\|_{L^n}.
$$
(Such a theorem would have been trivial by Hodge decomposition, if $\dot{W}^{1,n}(\mathbb{R}^n)$ were to embed into $L^{\infty}(\mathbb{R}^n)$.) The existing proofs of the above theorem are all long and complicated. On the contrary, a weaker version of this theorem, where one replaces the space $\dot{W}^{1,n} \cap L^{\infty}$ by $L^{\infty}$, can be obtained from the following theorem of Van Schaftingen \cite{MR2078071}, when $\ell \leq n-2$:

\begin{thm}[Van Schaftingen \cite{MR2078071}] \label{thm:vs}
Suppose $f$ is a smooth vector field on $\mathbb{R}^n$, with $$\operatorname{div}\, f = 0.$$ Then for any compactly supported smooth vector field $\phi$ on $\mathbb{R}^n$, we have
\begin{equation} \label{eq:mainest}
\left| \int_{\mathbb{R}^n} \langle f , \phi \rangle \right| \leq C \|f\|_{L^1} \|\nabla \phi\|_{L^n},
\end{equation}
where $\langle \cdot, \cdot \rangle$ is the pointwise Euclidean inner product of two vector fields in $\mathbb{R}^n$.
\end{thm}

See e.g.\ Bourgain and Brezis \cite{MR2293957}, Lanzani and Stein \cite{MR2122730}. We refer the interested reader to the survey in \cite{MR3298002}, for a more detailed account of this circle of ideas.

The original direct proof of Theorem \ref{thm:vs} in \cite{MR2078071} proceeds by decomposing
\[
 \int_{\mathbb{R}^n}  \langle f , \phi \rangle
 = \sum_{i= 1}^m \int_{\mathbb{R}} \Bigl(\int_{\mathbb{R}^{i - 1} \times \{s\} \times \mathbb{R}^{n - i}}
 f_i \phi_i\Bigr)\dif s,
\]
and by estimating first directly the innermost \((n-1)\)--dimensional integral. 
This gives the impression that the strategy is quite rigid. 
The first goal of this note is to prove Theorem~\ref{thm:vs} by averaging
a suitable estimate over all \emph{unit spheres} in $\mathbb{R}^n$. 

In a second part of this paper, we adapt this idea of averaging over families of sets 
to prove an analogue of Theorem~\ref{thm:vs}, in the setting where
$\mathbb{R}^n$ is replaced by the real hyperbolic space
$\mathbb{H}^n$:

\begin{thm} \label{thm:hyp}
Suppose $f$ is a smooth vector field on $\mathbb{H}^n$, with $$\operatorname{div}_g \, f = 0$$
where $\operatorname{div}_g$ is the divergence with respect to the metric $g$ on $\mathbb{H}^n$. Then for any compactly supported smooth vector field $\phi$ on $\mathbb{H}^n$, we have
\begin{equation} \label{eq:mainestHn}
\left| \int_{\mathbb{H}^n}
\langle f,\phi \rangle_g \dif V_g \right| \leq C \|f\|_{L^1(\mathbb{H}^n)} \|\nabla_g \phi\|_{L^n(\mathbb{H}^n)}.
\end{equation}
where $\langle  \cdot, \cdot \rangle_g$ and $\dif V_g$ are the pointwise inner product and the volume measure with respect to $g$ respectively, $\nabla_g \phi$ is the $(1,1)$ tensor given by the Levi-Civita connection of $\phi$ with respect to $g$, and
$$\|f\|_{L^1(\mathbb{H}^n)} = \int_{\mathbb{H}^n} |f|_g \dif V_g, \quad \|\nabla_g \phi\|_{L^n(\mathbb{H}^n)} = \left( \int_{\mathbb{H}^n} |\nabla_g \phi|_g^n \dif V_g \right)^{1/n}.$$
\end{thm}

We note that the above theorem is formulated entirely geometrically on $\mathbb{H}^n$, without the need of specifying a choice of coordinate chart. 
As explained in Appendix~\ref{appendix}, Theorem~\ref{thm:hyp} can be proved indirectly by patching together known estimates on $\mathbb{R}^n$ via a partition of unity, and by applying Hardy's inequality to get rid of lower order terms.

We shall prove Theorem~\ref{thm:hyp} by averaging a suitable estimate over a family of hypersurfaces in $\mathbb{H}^n$, where the family of hypersurfaces is obtained from the orbit of a ``vertical hyperplane'' under all isometries in $\mathbb{H}^n$. The latter shares a similar flavour to the proof we will give below of Theorem~\ref{thm:vs}. The main innovation there is in deducing Theorem~\ref{thm:hyp} from Proposition~\ref{prop:Hn}, and in establishing Lemma~\ref{lem:decompHmsimple} (see Section~\ref{sect:Hn} for details).
\\

\textbf{Acknowledgments.} S. Chanillo was partially supported by NSF grant
DMS 1201474. J. Van Schaftingen was partially supported by the Fonds de la
Recherche Scientifique-FNRS. P.-L. Yung was  partially supported by
a Titchmarsh Fellowship at the University of Oxford, a junior
research fellowship at St. Hilda's College, a direct grant for research from the Chinese University of Hong Kong (3132713), and 
an Early Career Grant CUHK24300915 from the Hong Kong Research Grant Council.

\section{Another proof of Theorem~\ref{thm:vs}}

Theorem~\ref{thm:vs} will follow from the following Proposition:

\begin{prop} \label{prop:Rn}
Let $f$, $\phi$ be as in Theorem~\ref{thm:vs}. Write $\mathbb{B}^n$ for the unit ball $\{x \in \mathbb{R}^n \colon |x| < 1\}$ in $\mathbb{R}^n$, and $\mathbb{S}^{n-1}$ for the unit sphere (i.e.\ the boundary of $\mathbb{B}^n$). Also write $\dif \sigma$ for the standard surface measure on $\mathbb{S}^{n-1}$, and $\nu$ for the outward unit normal to the sphere $\mathbb{S}^{n-1}$. Then
\begin{equation} \label{eq:sphere}
\left| \int_{\mathbb{S}^{n-1}} \langle f, \nu \rangle \langle \phi, \nu \rangle \dif \sigma \right|
\leq C \|f\|_{L^1(\mathbb{R}^n \setminus \mathbb{B}^n)}^{\frac{1}{n}} \|f\|_{L^1(\mathbb{S}^{n-1})}^{1-\frac{1}{n}} \|\phi \|_{W^{1,n}(\mathbb{S}^{n-1})}.
\end{equation}
Here $\|\phi\|_{W^{1,n}(\mathbb{S}^{n-1})} = \|\phi\|_{L^n(\S)} + \|\nabla_{\mathbb{S}^{n-1}} \phi\|_{L^n(\S)},$ where $\nabla_{\mathbb{S}^{n-1}} \phi$ is the (1,1) tensor on $\mathbb{S}^{n-1}$ given by the covariant derivative of the vector field $\phi$.
\end{prop}

The proof of Proposition~\ref{prop:Rn} in turn depends on the following two lemmas. The first one is a simple lemma about integration by parts:

\begin{lem} \label{lem:parts}
Let $f, \nu$ be as in Proposition~\ref{prop:Rn}. Then for any compactly supported smooth function $\psi$ on $\mathbb{R}^n$, we have
$$
\int_{\mathbb{S}^{n-1}} \langle f, \nu \rangle \psi \, \dif \sigma = - \int_{\mathbb{R}^n \setminus \mathbb{B}^n} \langle f, \nabla \psi \rangle\dif x.
$$
\end{lem}

The second one is a decomposition lemma for functions on the sphere $\mathbb{S}^{n-1}$:

\begin{lem} \label{lem:decomp}
Let $\varphi$ be a smooth function on $\mathbb{S}^{n-1}$. For any $\lambda > 0$, there exists a decomposition $$\varphi = \varphi_1 + \varphi_2 \quad \text{on $\mathbb{S}^{n-1}$},$$ and an extension $\tilde{\varphi}_2$ of $\varphi_2$ to $\mathbb{R}^n \setminus \mathbb{B}^n$, such that $\tilde{\varphi}_2$ is smooth and bounded on $\mathbb{R}^n \setminus \mathbb{B}^n$, with
$$
\|\varphi_1\|_{L^{\infty}(\mathbb{S}^{n-1})} \leq C \lambda^{\frac{1}{n}} \|\nabla_{\mathbb{S}^{n-1}} \varphi\|_{L^n (\mathbb{S}^{n-1})},
$$
$$
\|\nabla \tilde{\varphi}_2\|_{L^\infty (\mathbb{R}^n\setminus \mathbb{B}^n)}  \leq C \lambda^{\frac{1}{n}-1} \|\nabla_{\mathbb{S}^{n-1}} \varphi\|_{L^n (\mathbb{S}^{n-1})}.
$$
\end{lem}

Here $|\nabla_{\mathbb{S}^{n-1}} \varphi|$ is the norm of the gradient of the function $\varphi$ on $\mathbb{S}^{n-1}$.
We postpone the proofs of Lemmas~\ref{lem:parts} and \ref{lem:decomp} to the end of this section.

Now we are ready for the proof of Proposition~\ref{prop:Rn}.

\begin{proof}[Proof of Proposition~\ref{prop:Rn}]
Let $f, \phi$ be as in the statement of Theorem~\ref{thm:vs}. 
Apply Lemma~\ref{lem:decomp} to $\varphi = \langle \phi, \nu \rangle$, where $\lambda > 0$ is to be chosen. Then since
$$
\|\nabla_{\mathbb{S}^{n-1}} \langle \phi, \nu \rangle\|_{L^n(\mathbb{S}^{n-1})} \leq C \|\phi\|_{W^{1,n}(\mathbb{S}^{n-1})},
$$
there exists a decomposition
$$
\langle \phi, \nu \rangle = \varphi_1 + \varphi_2 \quad \text{on $\S$},
$$
and an extension $\tilde{\varphi}_2$ of $\varphi_2$ to $\Rn  \setminus \mathbb{B}^n$, such that $\tilde{\varphi}_2 \in C^{\infty} \cap L^{\infty}(\mathbb{R}^n \setminus \mathbb{B}^n)$,
$$
\|\varphi_1\|_{L^{\infty}(\mathbb{S}^{n-1})} \leq C \lambda^{\frac{1}{n}} \|\phi\|_{W^{1,n}(\mathbb{S}^{n-1})},
$$
and
$$
\|\nabla \tilde{\varphi}_2\|_{L^{\infty}(\mathbb{R}^n \setminus \mathbb{B}^n)} \leq C \lambda^{\frac{1}{n}-1} \|\phi\|_{W^{1,n}(\mathbb{S}^{n-1})}.
$$
Now 
\begin{align*}
\int_{\mathbb{S}^{n-1}} \langle f, \nu \rangle \langle \phi, \nu \rangle \dif \sigma
&=
\int_{\mathbb{S}^{n-1}} \langle f, \nu \rangle \varphi_1 \dif \sigma + \int_{\mathbb{S}^{n-1}} \langle f, \nu \rangle \varphi_2 \dif \sigma \\
&= I + I\!I.
\end{align*}
In the first term, we estimate trivially
$$
|I| \leq \|f\|_{L^1(\S)} \|\varphi_1\|_{L^{\infty}(\S)} \leq C \lambda^{\frac{1}{n}} \|f\|_{L^1(\S)} \|\phi\|_{W^{1,n}(\S)}.
$$
To estimate the second term, we let $\theta$ be a smooth cut-off function with compact support on $\mathbb{R}^n$ such that $\theta(x) = 1$ whenever $|x| \leq 1$.
For $\varepsilon \in (0, 1)$, let $\theta_{\varepsilon}(x) = \theta(\varepsilon x)$. 
Then $\theta_{\varepsilon} = 1$ on $\mathbb{S}^{n-1}$, so we can rewrite $I\!I$ as
$$
I\!I = \int_{\mathbb{S}^{n-1}} \langle f, \nu \rangle \varphi_2 \theta_{\varepsilon} \dif \sigma
$$
for any $\varepsilon \in (0,1)$. We then integrate by parts using Lemma~\ref{lem:parts}, with $\psi := \tilde{\varphi}_2 \theta_{\varepsilon}$, and obtain
$$
I\!I = -\int_{\mathbb{R}^n \setminus \mathbb{B}^n} \langle f, \nabla \tilde{\varphi}_2 \rangle \theta_{\varepsilon} \dif x - \int_{\mathbb{R}^n \setminus \mathbb{B}^n} \langle f, \nabla  \theta_{\varepsilon} \rangle  \tilde{\varphi}_2 \dif x.
$$
(The cut-off function $\theta_{\varepsilon}$ is inserted so that $\psi$ has compact support.) We now let $\varepsilon \to 0^+$. The second term then tends to $0$, since it is just
$$
-\varepsilon \int_{\mathbb{R}^n \setminus \mathbb{B}^n} \langle f(x), (\nabla  \theta)(\varepsilon x) \rangle  \tilde{\varphi}_2(x) \dif x,
$$
where $f \in L^1$,  $\nabla \theta(\varepsilon \cdot) \in L^{\infty}$ and  $\tilde{\varphi}_2 \in L^{\infty}$ on $\mathbb{R}^n \setminus \mathbb{B}^n$. On the other hand, the first term tends to 
$$
-\int_{\mathbb{R}^n \setminus \mathbb{B}^n} \langle f, \nabla \tilde{\varphi}_2 \rangle \dif x
$$
by dominated convergence theorem, since $f \in L^1$ and  $\nabla \tilde{\varphi}_2 \in L^{\infty}$ on $\mathbb{R}^n \setminus \mathbb{B}^n$. 
As a result,
$$
I\!I = -\int_{\mathbb{R}^n \setminus \mathbb{B}^n} \langle f, \nabla \tilde{\varphi}_2 \rangle \dif x,
$$
from which we see that
$$
|I\!I| \leq \|f\|_{L^1(\mathbb{R}^n \setminus \mathbb{B}^n)} \|\nabla \tilde{\varphi}_2\|_{L^{\infty}(\mathbb{R}^n \setminus \mathbb{B}^n)} \leq C \lambda^{\frac{1}{n}-1} \|f\|_{L^1(\mathbb{R}^n \setminus \mathbb{B}^n)} \|\phi\|_{W^{1,n}(\S)}.
$$
Together, by choosing $\lambda = \frac{\|f\|_{L^1(\mathbb{R}^n \setminus \mathbb{B}^n)}}{\|f\|_{L^1(\S)}}$, we get
$$
\left| \int_{\mathbb{S}^{n-1}} \langle f, \nu \rangle \langle \phi, \nu \rangle \dif \sigma \right| \leq C \|f\|_{L^1(\mathbb{R}^n \setminus \mathbb{B}^n)}^{\frac{1}{n}} \|f\|_{L^1(\S)}^{1-\frac{1}{n}} \|\phi\|_{W^{1,n}(\S)}
$$
as desired.
\end{proof}

We will now deduce Theorem~\ref{thm:vs} from Proposition~\ref{prop:Rn}. The idea is to average \eqref{eq:sphere} over all unit spheres in $\mathbb{R}^n$.

\begin{proof}[Proof of Theorem~\ref{thm:vs}]
First, for each fixed $x \in \mathbb{R}^n$, we have
\begin{equation} \label{eq:sphavg}
\langle f(x), \phi(x) \rangle = c \int_{\S} \langle f(x), \omega \rangle \langle \phi(x), \omega \rangle \dif \sigma(\omega)
\end{equation}
where we are identifying $\omega \in \S$ with the corresponding unit tangent vector to $\mathbb{R}^n$ based at the point $x$. Hence to estimate $\int_{\Rn} \langle f(x), \phi(x) \rangle\dif x$, it suffices to estimate
$$
\int_{\Rn} \int_{\S} \langle f(x), \omega \rangle \langle \phi(x), \omega \rangle \dif \sigma(\omega)\dif x,
$$
which is the same as
$$
\int_{\Rn} \int_{\S} \langle f(z+\omega), \omega \rangle \langle \phi(z+\omega), \omega \rangle \dif \sigma(\omega)\dif z
$$
by a change of variable $(x,\omega) \mapsto (z+\omega,\omega)$. Now when $z = 0$, the inner integral can be estimated by Proposition~\ref{prop:Rn}; for a general $z \ne 0$, one can still estimate the inner integral by Proposition~\ref{prop:Rn}, since the proposition is invariant under translations. Thus the above double integral is bounded, in absolute value, by
$$
C \|f\|_{L^1(\Rn)}^{\frac{1}{n}} \int_{\Rn} \|f(z+\cdot)\|_{L^1(\S)}^{1-\frac{1}{n}} \|\phi(z+\cdot)\|_{W^{1,n}(\S)}\dif z.
$$
Applying H\"older's inequality to the last integral in $z$, one bounds this by
$$
C \|f\|_{L^1(\Rn)}^{\frac{1}{n}} \left( \int_{\Rn} \|f(z+\cdot)\|_{L^1(\S)}\dif z \right)^{1-\frac{1}{n}} \left( \int_{\Rn} \|\phi(z+\cdot)\|_{W^{1,n}(\S)}^n\dif z \right)^{\frac{1}{n}}.
$$
Since
$$
\int_{\Rn} \|f(z+\cdot)\|_{L^1(\S)}\dif z = c \|f\|_{L^1(\Rn)}
$$
and
$$
\int_{\Rn} \|\phi(z+\cdot)\|_{W^{1,n}(\S)}^n\dif z
\leq c \left( \|\phi\|_{L^n(\Rn)}^n + \|\nabla \phi\|_{L^n(\Rn)}^n \right),
$$
we proved that under the assumption of Theorem~\ref{thm:vs}, we have
\begin{equation} \label{eq:lowerorder}
\left| \int_{\Rn} \langle f, \phi \rangle \right| \leq C \|f\|_{L^1(\Rn)} \left( \|\nabla \phi\|_{L^n(\Rn)} + \|\phi\|_{L^n(\Rn)} \right).
\end{equation}
This is almost the desired conclusion, except that we have an additional zeroth order term on $\phi$ on the right hand side of the estimate. But that can be scaled away by homogeneity. In fact, if $f$ and $\phi$ satisfies the assumption of Theorem~\ref{thm:vs}, then so does the dilations
$$
f_{\varepsilon}(x) := \varepsilon^{-n} f(\varepsilon^{-1} x), \quad \phi_{\varepsilon}(x) := \phi(\varepsilon^{-1} x), \quad \varepsilon > 0.
$$
Applying \eqref{eq:lowerorder} to $f_{\varepsilon}$ and $\phi_{\varepsilon}$ instead, we get
$$
\left| \int_{\Rn} \langle f_{\varepsilon}, \phi_{\varepsilon} \rangle \right| \leq C \|f_{\varepsilon}\|_{L^1(\Rn)} \left( \|\nabla \phi_{\varepsilon}\|_{L^n(\Rn)} + \|\phi_{\varepsilon}\|_{L^n(\Rn)} \right),
$$
i.e.
$$
\left| \int_{\Rn} \langle f, \phi \rangle \right| \leq C \|f\|_{L^1(\Rn)} \left( \|\nabla \phi\|_{L^n(\Rn)} + \varepsilon \|\phi\|_{L^n(\Rn)} \right),
$$
so letting $\varepsilon \to 0^+$, we get the desired conclusion of Theorem~\ref{thm:vs}.
\end{proof}

\begin{proof}[Proof of Lemma~\ref{lem:parts}]
Note that $\langle f, \nu \rangle \psi = \langle \psi f, \nu \rangle$,  and $\nu$ is the inward unit normal to $\partial (\mathbb{R}^n \setminus \mathbb{B}^n)$. So by the divergence theorem on $\mathbb{R}^n$, we have
$$
\int_{\mathbb{S}^{n-1}} \langle f, \nu \rangle \psi \, \dif \sigma = -\int_{\mathbb{R}^n \setminus \mathbb{B}^n} \operatorname{div} (\psi f)\dif x.
$$
But  since $\operatorname{div} f = 0$, we have $$\operatorname{div} (\psi f) = \langle f, \nabla \psi \rangle + \psi\, \operatorname{div} f = \langle f, \nabla \psi \rangle,$$ and the desired equality follows.
\end{proof}

\begin{proof}[Proof of Lemma~\ref{lem:decomp}]
Suppose $\varphi$ and $\lambda$ are as in Lemma~\ref{lem:decomp}. We will construct first a decomposition $\varphi = \varphi_1 + \varphi_2$ on $\mathbb{S}^{n-1}$, so that both $\varphi_1$ and $\varphi_2$ are smooth on $\mathbb{S}^{n-1}$, and 
$$
\|\varphi_1\|_{L^{\infty}(\mathbb{S}^{n-1})} \leq C \lambda^{\frac{1}{n}} \|\nabla_{\mathbb{S}^{n-1}} \varphi\|_{L^n (\mathbb{S}^{n-1})},
$$
$$
\|\nabla_{\mathbb{S}^{n-1}} \varphi_2\|_{L^\infty (\mathbb{S}^{n-1})}  \leq C \lambda^{\frac{1}{n}-1} \|\nabla_{\mathbb{S}^{n-1}} \varphi\|_{L^n (\mathbb{S}^{n-1})}.
$$
(Here $\nabla_{\mathbb{S}^{n-1}} \varphi_2$ is the gradient on $\mathbb{S}^{n-1}$.) Once this is established, the lemma will follow, by extending $\varphi_2$ so that it is homogeneous of degree 0; in other words, we will then define
$$
\tilde{\varphi}_2(x) := \varphi_2 \left( \frac{x}{|x|} \right), \quad |x| \geq 1.
$$
It is then straight forward to verify that
$$
\|\nabla \tilde{\varphi}_2\|_{L^\infty (\mathbb{R}^n\setminus \mathbb{B}^n)}  \leq C \lambda^{\frac{1}{n}-1} \|\nabla_{\mathbb{S}^{n-1}} \varphi\|_{L^n (\mathbb{S}^{n-1})},
$$
since the radial derivative of $\tilde{\varphi}_2$ is zero.

To construct the desired decomposition on $\mathbb{S}^{n-1}$, we proceed as follows.

If \(\lambda \ge 1\), we set \(\varphi_2 = \fint_{\mathbb{S}^n} \varphi\), so that $\nabla_{\mathbb{S}^{n-1}} \varphi_2 = 0$ on $\mathbb{S}^{n-1}$; then $$\varphi_1 = \varphi - \fint_{\mathbb{S}^n} \varphi,$$ and the estimate for $\varphi_1$ follows from the classical Morrey--Sobolev estimate.

If \(0 < \lambda < 1\), we pick a non-negative radial cut-off function $\eta \in C^{\infty}_c(\mathbb{R}^n)$, with $\eta = 1$ in a neighborhood of $0$, and define $\eta_{\lambda}(x) = \eta(\lambda^{-1} x)$ for $x \in \mathbb{R}^n$. We then consider the function 
$$
x \in \mathbb{R}^{n} \mapsto \int_{\mathbb{S}^{n-1}} \eta_{\lambda}(x-y) \dif \sigma(y).
$$
When restricted to $x \in \mathbb{S}^{n-1}$, this function is a constant independent of the choice of $x \in \mathbb{S}^{n-1}$, by rotation invariance of the integral. We then write $c_{\lambda}$ for this constant, i.e.
\begin{equation} \label{eq:clamdef}
c_{\lambda} := \int_{\mathbb{S}^{n-1}} \eta_{\lambda}(x-y) \dif \sigma(y), \quad |x| = 1.
\end{equation}
Note that by our choice of $\eta$, when $0 < \lambda < 1$, 
\begin{equation} \label{eq:clamest}
c_{\lambda} \simeq \lambda^{n-1}.
\end{equation}
Now we define, for $x \in \mathbb{S}^{n-1}$, that
$$
\varphi_2(x) = c_{\lambda}^{-1} \int_{\mathbb{S}^{n-1}} \eta_{\lambda}(x-y) \varphi(y) \dif \sigma(y), \quad \varphi_1(x) = \varphi(x) - \varphi_2(x).
$$
Then for $x \in \mathbb{S}^{n-1}$, we have
$$
\varphi_1(x) = c_{\lambda}^{-1} \int_{\mathbb{S}^{n-1}} \eta_{\lambda}(x-y) [\varphi(x) - \varphi(y)] \dif \sigma(y)
$$
by definition of $c_{\lambda}$. But for $x, y \in \mathbb{S}^{n-1}$, we have, by Morrey's embedding, that
$$
|\varphi(x)-\varphi(y)| \leq C |x-y|^{\frac{1}{n}} \|\nabla_{\mathbb{S}^{n-1}} \varphi\|_{L^n(\mathbb{S}^{n-1})}.
$$
It follows that
$$
|\varphi_1(x)| \leq  C  \|\nabla_{\mathbb{S}^{n-1}} \varphi\|_{L^n(\mathbb{S}^{n-1})}
c_{\lambda}^{-1} \int_{\mathbb{S}^{n-1}} \eta_{\lambda}(x-y) |x-y|^{\frac{1}{n}} \dif \sigma(y)
$$
Letting $\tilde{\eta}(x) =  |x|^{\frac{1}{n}}\eta(x)$, and $\tilde{\eta}_{\lambda}(x) = \tilde{\eta}(\lambda^{-1} x)$, we see that the right hand side above is just
$$
C \lambda^{\frac{1}{n}} \|\nabla_{\mathbb{S}^{n-1}} \varphi\|_{L^n(\mathbb{S}^{n-1})}
c_{\lambda}^{-1} \int_{\mathbb{S}^{n-1}} \tilde{\eta}_{\lambda}(x-y) \dif \sigma(y).
$$
But this last integral can be estimated by
$$
\int_{\mathbb{S}^{n-1}} \tilde{\eta}_{\lambda}(x-y) \dif \sigma(y) \lesssim \lambda^{n-1},
$$
by the support and $L^{\infty}$ bound of $\tilde{\eta}_{\lambda}$. Hence using also (\ref{eq:clamest}), we see that
$$
\|\varphi_1\|_{L^{\infty}(\mathbb{S}^{n-1})} \leq  C \lambda^{\frac{1}{n}} \|\nabla_{\mathbb{S}^{n-1}} \varphi\|_{L^n(\mathbb{S}^{n-1})}
$$
as desired.

Next, suppose $x \in \mathbb{S}^{n-1}$, and $v$ is a unit tangent vector to $\mathbb{S}^{n-1}$ at $x$. Then 
\begin{equation} \label{eq:partialvphi2y}
\partial_v \phi_2(x) = \lambda^{-1} c_{\lambda}^{-1}   \int_{\mathbb{S}^{n-1}} \langle v, \nabla \eta \rangle (\lambda^{-1}(x-y)) \varphi(y) \dif \sigma(y).
\end{equation}
But if we differentiate both sides of the definition (\ref{eq:clamdef}) of $c_{\lambda}$ with respect to $\partial_v$, we see that
\begin{equation} \label{eq:partialvphi2x}
0 = \int_{\mathbb{S}^{n-1}} \langle v, \nabla \eta \rangle (\lambda^{-1}(x-y)) \dif \sigma(y).
\end{equation}
Multiplying (\ref{eq:partialvphi2x}) by $\lambda^{-1} c_{\lambda}^{-1} \phi(x)$, and subtracting that from (\ref{eq:partialvphi2y}), we get
$$
\partial_v \phi_2(x) = \lambda^{-1}  c_{\lambda}^{-1}  \int_{\mathbb{S}^{n-1}} \langle v, \nabla \eta \rangle (\lambda^{-1}(x-y)) [\varphi(y)-\varphi(x)] \dif \sigma(y).
$$
Using Morrey's embedding again, we see that
$$
|\partial_v \phi_2(x)| \leq C \lambda^{-1} \|\nabla_{\mathbb{S}^{n-1}} \varphi\|_{L^n(\mathbb{S}^{n-1})} c_{\lambda}^{-1}  \int_{\mathbb{S}^{n-1}} |\langle v, \nabla \eta \rangle| (\lambda^{-1}(x-y)) |x-y|^{\frac{1}{n}} \dif \sigma(y).
$$
Letting $\bar{\eta}(x) =  |x|^{\frac{1}{n}}|\langle v, \nabla \eta \rangle|(x)$, and $\bar{\eta}_{\lambda}(x) = \bar{\eta}(\lambda^{-1} x)$, we see that the right hand side above is just
$$
C \lambda^{\frac{1}{n}} \lambda^{-1} \|\nabla_{\mathbb{S}^{n-1}} \varphi\|_{L^n(\mathbb{S}^{n-1})}
c_{\lambda}^{-1} \int_{\mathbb{S}^{n-1}} \bar{\eta}_{\lambda}(x-y) \dif \sigma(y).
$$
But this last integral can be estimated by
$$
\int_{\mathbb{S}^{n-1}} \bar{\eta}_{\lambda}(x-y) \dif \sigma(y) \dif \sigma(y) \lesssim \lambda^{n-1},
$$
by the support and $L^{\infty}$ bound of $\bar{\eta}_{\lambda}$. Hence using also (\ref{eq:clamest}), we see that
$$
\|\nabla_{\mathbb{S}^{n-1}} \varphi_2\|_{L^{\infty}(\mathbb{S}^{n-1})} \leq  C \lambda^{\frac{1}{n}-1} \|\nabla_{\mathbb{S}^{n-1}} \varphi\|_{L^n(\mathbb{S}^{n-1})}
$$
as desired.
\end{proof}

\section{A borderline Sobolev embedding on the real hyperbolic space $\mathbb{H}^n$} \label{sect:Hn}

We now turn to a corresponding result on the real hyperbolic space $\mathbb{H}^n$. We will first give a direct proof in this current section, in the spirit of the earlier proof of Theorem~\ref{thm:vs} using spherical averages. In the appendix we give a less direct proof, using a variant of Theorem~\ref{thm:vs} on $\mathbb{R}^n$.

First we need some notations. We will use the upper half space model for the hyperbolic space. In other words, we take $\mathbb{H}^n$ to be
$$
\mathbb{H}^n =  \mathbb{R}^n_+ = \{x = (x',x_n) \in \mathbb{R}^{n-1} \times \mathbb{R} \colon x_n > 0 \}
$$
and the metric on $\mathbb{H}^n$ to be
$$
g:=\frac{\vert \dif x \vert^2}{x_n^2}.
$$
We will use the following orthonormal frame of vector fields
$$
e_i := x_n \frac{\partial}{\partial x_i}, \quad i = 1, \dotsc, n,
$$
at every point of $\mathbb{H}^n$. Note that if $j \ne n$, then
\begin{equation} \label{eq:covar_nj}
\nabla_{e_n} e_j = 0.
\end{equation}
(Here $\nabla = \nabla_g$ is the Levi-Civita connection with respect to the hyperbolic metric $g$.) In fact, since $\{e_1,\dots,e_n\}$ is an orthonormal basis, for any $k=1,\dots,n$, we have
\begin{align*}
\langle \nabla_{e_n} e_j, e_k \rangle_g
&= \frac{1}{2} \left( \langle [e_n,e_j], e_k \rangle_g - \langle [e_n,e_k], e_j \rangle_g  - \langle [e_j,e_k], e_n \rangle_g \right) \\
&= \frac{1}{2} \left( \langle e_j, e_k \rangle_g - \langle (1-\delta_{kn}) e_k, e_j \rangle_g  - \langle -\delta_{kn} e_j, e_n \rangle_g \right) =0.
\end{align*}
Also, we have
\begin{equation} \label{eq:covar_nn}
\nabla_{e_n} e_n = 0.
\end{equation}
This is because if $j \ne n$, then
\begin{align*}
\langle \nabla_{e_n} e_n, e_j \rangle_g
&= - \langle e_n, \nabla_{e_n} e_j \rangle_g =0
\end{align*}
by \eqref{eq:covar_nj}, and
$$
\langle \nabla_{e_n} e_n, e_n \rangle_g
= \frac{1}{2} e_n \left( \langle e_n, e_n \rangle_g \right) = 0.
$$

To prove Theorem~\ref{thm:hyp}, note that we only need to consider the case $n \geq 2$, since when $n=1$, $$\|\phi\|_{L^{\infty}(\mathbb{H}^1)} \leq \int_0^{\infty} |\partial_y \phi(y)| dy = \|\nabla_g \phi\|_{L^1(\mathbb{H}^1)},$$ and \eqref{eq:mainestHn} follows trivially. Hence from now on we assume $n \geq 2$.

We will deduce Theorem~\ref{thm:hyp} from the following Proposition:

\begin{prop} \label{prop:Hn}
Assume $n \geq 2$. Let $f$, $\phi$ be as in Theorem~\ref{thm:hyp}. Write $S$ for the copy of $(n-1)$-dimensional hyperbolic space inside $\mathbb{H}^n$, given by
$$
S = \{x \in \mathbb{H}^n \colon x_1 = 0\},
$$
and $X$ for the half-space
$$
X = \{x \in \mathbb{H}^n \colon x_1 > 0\}
$$
so that $S$ is the boundary of $X$.
Also write $\dif V'_g$ for the volume measure on $S$ with respect to the hyperbolic metric on $S$, and $\nu = e_1$ for the unit normal to $S$. Then
\begin{equation} \label{eq:HnS}
\left| \int_{S}  \langle f, \nu \rangle_g \langle \phi, \nu \rangle_g \dif V'_g \right| \leq C \|f\|_{L^1(X)}^{\frac{1}{n}} \|f\|_{L^1(S)}^{1-\frac{1}{n}} \|\phi\|_{W^{1,n}(S)}.
\end{equation}
Here $\|\phi\|_{W^{1,n}(S)} = \|\phi\|_{L^n(S)} + \|\nabla_g \phi\|_{L^n(S)}$, and all integrals on $S$ on the right hand side are with respect to $\dif V'_g$.
\end{prop}

The $S$ will be called a vertical hyperplane in $\mathbb{H}^n$. It is a totally geodesic submanifold of $\mathbb{H}^n$. We will consider all hyperbolic hyperplanes in $\mathbb{H}^n$, that is the image of $S$ under all isometries of $\mathbb{H}^n$. The set of all such hypersurfaces in $\mathbb{H}^n$ will be denoted by $\mathcal{S}$; it will consist of all Euclidean parallel translates of $S$ in the $x'$-direction, and all Euclidean northern hemispheres whose centers lie on the plane $\{x_n=0\}$.

The proof of Proposition \ref{prop:Hn} in turn depends on the following two lemmas. The first one is a simple lemma about integration by parts, which is the counterpart of Lemma~\ref{lem:parts}:

\begin{lem} \label{lem:partsHn}
Assume $n \geq 2$. Let $f$, $S$, $X$, $\nu$ be as in Proposition~\ref{prop:Hn}. Then for any compactly supported smooth function $\psi$ on $\mathbb{H}^n$, we have
$$
\int_{S} \langle f, \nu \rangle_g \psi \,  \dif V'_g = -\int_X \langle f, \nabla_g \psi \rangle_g \,   \dif V_g.
$$
\end{lem}

The second one is a decomposition lemma for functions on $S$, which is  the counterpart of Lemma~\ref{lem:decomp}:

\begin{lem} \label{lem:decompHn}
Assume $n \geq 2$. Let $\varphi$ be a smooth function with compact support on $S$. For any $\lambda > 0$, there exists a decomposition
$$
\varphi = \varphi_1 + \varphi_2 \quad \text{on $S$},
$$
and an extension $\tilde{\varphi}_2$ of $\varphi_2$ to $\mathbb{H}^n$, such that $\tilde{\varphi}_2$ is smooth with compact support on $\mathbb{H}^n$, and
$$
\|\varphi_1\|_{L^{\infty}(S)} \leq C \lambda^{\frac{1}{n}} \|\nabla_g \varphi \|_{L^n(S)},
$$
with
$$
\|\nabla_g \tilde{\varphi}_2\|_{L^{\infty}(\mathbb{H}^n)} \leq C \lambda^{\frac{1}{n}-1}\|\nabla_g \varphi \|_{L^n(S)}.
$$
\end{lem}


We postpone the proofs of Lemmas~\ref{lem:partsHn} and \ref{lem:decompHn} to the end of this section.

Now we are ready for the proof of Proposition~\ref{prop:Hn}.

\begin{proof}[Proof of Proposition~\ref{prop:Hn}]
Let $f, \phi$ be as in the statement of Theorem~\ref{thm:hyp}. Apply Lemma~\ref{lem:decompHn} to $\varphi = \langle \phi, \nu \rangle_g$, where $\lambda > 0$ is to be chosen. Then since
$$
\|\nabla_g \langle \phi, \nu \rangle_g \|_{L^n(S)} \leq C \|\phi\|_{W^{1,n}(S)},
$$
(this follows since $|e_k \langle \phi, \nu \rangle| = |\langle \nabla_{e_k} \phi, \nu \rangle + \langle \phi, \nabla_{e_k} \nu \rangle| \leq |\nabla_g \phi|_g + |\phi|_g$ for all $k$), there exists a decomposition
$$
\langle \phi, \nu \rangle_g = \varphi_1 + \varphi_2 \quad \text{on $S$},
$$
and an extension $\tilde{\varphi}_2$ of $\varphi_2$ to $\mathbb{H}^n$, such that $\tilde{\varphi}_2 \in C^{\infty}_c(\mathbb{H}^n)$,
$$
\|\varphi_1\|_{L^{\infty}(S)} \leq C \lambda^{\frac{1}{n}} \|\phi\|_{W^{1,n}(S)},
$$
and
$$
\|\nabla_g \tilde{\varphi}_2\|_{L^{\infty}(\mathbb{H}^n)} \leq C \lambda^{\frac{1}{n}-1} \|\phi\|_{W^{1,n}(S)}.
$$
Now
\begin{align*}
\int_{S} \langle f, \nu \rangle_g \langle \phi, \nu \rangle_g \, \dif V'_g
&=
\int_{S} \langle f, \nu \rangle_g \varphi_1 \, \dif V'_g  + \int_{S} \langle f, \nu \rangle_g \varphi_2 \, \dif V'_g  \\
&= I + I\!I.
\end{align*}
In the first term, we estimate trivially
$$
|I| \leq \|f\|_{L^1(S)} \|\varphi_1\|_{L^{\infty}(S)} \leq C \lambda^{\frac{1}{n}} \|f\|_{L^1(S)} \|\phi\|_{W^{1,n}(S)}.
$$
In the second term, we first integrate by parts using Lemma~\ref{lem:partsHn}, with $\psi = \tilde{\varphi}_2$, and obtain
$$
I\!I = -\int_{X} \langle f, \nabla_g \tilde{\varphi}_2 \rangle_g \, \dif V_g,
$$
so
\begin{align*}
|I\!I| &\leq \|f\|_{L^1(X)}  \|\nabla_g \tilde{\varphi}_2\|_{L^{\infty}(\mathbb{H}^n)}
\\
&\leq C \lambda^{\frac{1}{n}-1} \|f\|_{L^1(X)} \|\phi\|_{W^{1,n}(S)}.
\end{align*}
Together, by choosing $\lambda = \frac{\|f\|_{L^1(X)}}{\|f\|_{L^1(S)}}$, we get
$$
\left| \int_{S} \langle f, \nu \rangle_g \langle \phi, \nu \rangle_g \dif V'_g \right| \leq C \|f\|_{L^1(X)}^{\frac{1}{n}} \|f\|_{L^1(S)}^{1-\frac{1}{n}} \|\phi\|_{W^{1,n}(S)}
$$
as desired.
\end{proof}

We will now deduce Theorem~\ref{thm:hyp} from Proposition~\ref{prop:Hn}. The idea is to average \eqref{eq:HnS} over all images of $S$ under isometries in $\mathbb{H}^n$ (i.e.\ all hypersurfaces in the collection $\mathcal{S}$).

\begin{proof}[Proof of Theorem~\ref{thm:hyp}]
First, for each fixed $x  = (x',x_n) \in \mathbb{H}^n$, we have the following analogue of the identity \eqref{eq:sphavg}, used in the proof of Theorem~\ref{thm:vs}:
$$
\langle f(x), \phi(x) \rangle_g = c \int_{\mathbb{S}^{n-1}} \langle f(x), x_n \omega \rangle_g \langle \phi(x), x_n \omega \rangle_g \, \dif \sigma(\omega)
$$
Here we are identifying $\omega \in \mathbb{S}^{n-1}$ with the corresponding tangent vector to $\mathbb{H}^n$ based at the point $x$. (Note then $x_n \omega$ has length 1 with respect to the metric $g$ at $x$, so $x_n \omega$ belongs to the unit sphere bundle at $x$.) Furthermore, since the above integrand is even in $\omega$, we may replace the integral over $\mathbb{S}^{n-1}$ by the integral only over the northern hemisphere $\mathbb{S}^{n-1}_+ := \{\omega \in \mathbb{S}^{n-1} \colon \omega_n > 0\}$. Hence to estimate $\int_{\mathbb{H}^n} \langle f(x), \phi(x) \rangle_g \dif V_g$, it suffices to estimate
\begin{equation} \label{eq:avHn}
\int_{\mathbb{R}^n_+} \int_{\mathbb{S}^{n-1}_+} \langle f(x), x_n \omega \rangle_g \langle \phi(x), x_n \omega \rangle_g \, \dif \sigma(\omega) \frac{\dif x}{x_n^n}.
\end{equation}
We will compute this integral by making a suitable change of variables.

To do so, given $x \in \mathbb{R}^n_+$ and $\omega \in \mathbb{S}^{n-1}_+$, 
let $S(x,\omega)$ be the hyperbolic hypersurface in $\mathcal{S}$ passing through $x$ with normal vector $\omega$ at $x$. In other words, $S(x,\omega)$ would be an Euclidean hemisphere, with center on the plane $\{x_n=0\}$; we denote the center of this Euclidean hemisphere by $(z,0)$, where $z = z(x,\omega)$.

For each fixed $x \in \mathbb{R}^n_+$, the map $\omega \mapsto z(x,\omega)$ provides an invertible change of variables from $\mathbb{S}^{n-1}_+$ to $\mathbb{R}^{n-1}$. Thus we are led to parametrize the integral in \eqref{eq:avHn} by \(z\) instead of \(\omega\). In order to do that we observe that the vectors \(x - (z, 0)\) and \(\omega\) are collinear. This implies that if $z = z(x,\omega)$, then
\[
 \omega = \frac{x - (z, 0)}{\vert x - (z, 0)\vert}.
\]
(Here $|x-(z,0)|$ is the Euclidean norm of $x-(z,0)$.) Write $\Phi_x (z)$ for the right hand side of the above equation. We view $\Phi_x$ as a map \(\Phi_x : \mathbb{R}^{n - 1} \to \mathbb{S}^{n - 1}_+ \subset \mathbb{R}^n \), and compute the Jacobian of the map.
We have 
\[
 (D \Phi_x)^t(z) =\frac{1}{\vert x - (z, 0)\vert} \Bigl(( -I, 0) + \frac{ (x' - z) \otimes (x - (z, 0))^t}{\vert x - (z, 0)\vert^2} \Bigr),
\]
(here we think of $x$, $z$ as column vectors, and $D\Phi_x$ as an $(n-1) \times n$ matrix). Thus 
\[
 (D \Phi_x)^t D \Phi_x(z)
 =\frac{1}{\vert x - (z, 0)\vert^2} \Bigl(I - \frac{(x' - z) \otimes (x' - z)^t}{\vert x - (z, 0)\vert^2}\Bigr).
\]
By computing the determinant in a basis that contains \(x' - z\), we get
\[
\begin{split}
 \operatorname{Jac} \Phi_x (z) &= \sqrt{\det [(D \Phi_x)^t D \Phi_x(z)]}\\
 &= \frac{1}{\vert x - (z, 0) \vert^{n - 1}} \sqrt{1 - \frac{\vert x' - z\vert^2}{\vert x - (z, 0)\vert^2}} \\
 &=\frac{x_n}{\vert x - (z, 0) \vert^n}.
\end{split}
\]
By a change of variable \(\omega = \Phi_x (z)\), and using Fubini's theorem, we see that (\ref{eq:avHn}) is equal to
\begin{equation} \label{eq:avHntemp}
\int_{\mathbb{R}^{n-1}} \int_{\mathbb{R}^n_+} \frac{\langle f(x), x_n \Phi_x (z) \rangle_g \langle \phi(x), x_n \Phi_x (z) \rangle_g}{\vert x - (z, 0)\vert^{n}} \, \frac{\dif x}{x_n^{n - 1}} \dif z.
\end{equation}
Now we fix $z \in \mathbb{R}^{n-1}$, and compute the inner integral over $x$ by integrating over successive hemispheres of radius $r$ centered at $(z,0)$. More precisely, let $S(z,r)$ be the Euclidean northern hemisphere with center $(z,0)$ and of radius $r > 0$. Then $S(z,r) \in \mathcal{S}$, and for any $z \in \mathbb{R}^{n-1}$, we have
\begin{align*}
&\int_{\mathbb{R}^n_+} \frac{\langle f(x), x_n \Phi_x (z) \rangle_g \langle \phi(x), x_n \Phi_x (z) \rangle_g}{\vert x - (z, 0)\vert^{n}} \, \frac{\dif x}{x_n^{n - 1}} \\
=& \int_0^{\infty} \int_{x \in S(z,r)} \langle f(x), x_n \Phi_x (z) \rangle_g \langle \phi(x), x_n \Phi_x (z) \rangle_g \, \frac{\dif \sigma(x)}{x_n^{n - 1}} \frac{\dif r}{r^n},
\end{align*}
where $\dif \sigma(x)$ is the Euclidean surface measure on $S(z,r)$. However, if $x \in S(z,r)$, then $x_n \Phi_x(z)$ is precisely the upward unit normal to $S(z,r)$ at $x$. Also, if $\dif V'_g$ is the induced surface measure on $S(z,r)$ from the hyperbolic metric on $\mathbb{H}^n$, then $$
\dif V'_g = \frac{ \dif \sigma(x)}{x_n^{n-1} r};
$$ 
indeed if we write $\omega = \frac{x-(z,0)}{|x-(z,0)|}$, then at $x \in S(z,r)$ we have
$$
\dif V'_g = i_{x_n \omega} \dif V_g = i_{x_n \omega} \frac{\dif x}{x_n^n} = \frac{i_{r\omega} \dif x}{x_n^{n-1} r} = \frac{\dif \sigma(x)}{x_n^{n-1} r}.
$$
(here $i$ denotes the interior product of a vector with a differential form). Hence the integral (\ref{eq:avHntemp}) is just equal to
\begin{equation} \label{eq:avHntemp2}
\int_{\mathbb{R}^{n-1}} \int_0^{\infty} \int_{S(z,r)} \langle f, \nu \rangle_g \langle \phi, \nu \rangle_g \, \dif V'_g \frac{\dif r}{r^{n-1}} \dif z.
\end{equation}
By Proposition~\ref{prop:Hn} and its invariance under isometries of the hyperbolic space \(\mathbb{H}^n\), we have 
\begin{align*}
&\left| \int_{S(z,r)} \langle f, \nu \rangle_g \langle \phi, \nu \rangle_g \, \dif V'_g  \right| \\
\leq & C \|f\|_{L^1(\mathbb{H}^n)}^{\frac{1}{n}} \left( \int_{S(z,r)} |f|_g \, \dif V'_g \right)^{1-\frac{1}{n}} \left( \int_{S(z,r)} \left( |\nabla_g \phi|_g^n + |\phi|_g^n \right) \, \dif V'_g \right)^{\frac{1}{n}}.
\end{align*}
Hence by H\"older's inequality, (\ref{eq:avHntemp2}) is bounded by
\begin{align*}
&C \|f\|_{L^1(\mathbb{H}^n)}^{\frac{1}{n}} \left(  \int_{\mathbb{R}^{n-1}} \int_0^{\infty}  \int_{S(z,r)} |f|_g \, \dif V'_g \frac{\dif r}{r^{n-1}} \dif z \right)^{1-\frac{1}{n}} \\
&\quad \quad \cdot \left(  \int_{\mathbb{R}^{n-1}} \int_0^{\infty}  \int_{S(z,r)} \left( |\nabla_g \phi|_g^n + |\phi|_g^n \right)  \, \dif V'_g \frac{\dif r}{r^{n-1}} \dif z \right)^{\frac{1}{n}}.
\end{align*}
But undoing our earlier changes of variable, we see that
\begin{align*}
&\int_{\mathbb{R}^{n-1}} \int_0^{\infty}  \int_{S(z,r)} |f|_g \, \dif V'_g \frac{\dif r}{r^{n-1}} \dif z \\
=& \int_{\mathbb{R}^{n-1}} \int_0^{\infty}  \int_{x \in S(z,r)} |f(x)|_g \, \frac{\dif \sigma(x)}{x_n^{n-1}} \frac{\dif r}{r^n} \dif z \\
=& \int_{\mathbb{R}^{n-1}} \int_{\mathbb{R}^n_+} \frac{|f(x)|_g}{|x-(z,0)|^n} \frac{\dif x}{x_n^{n-1}} \dif z \\
=&  \int_{\mathbb{R}^n_+} |f(x)|_g \left( \int_{\mathbb{R}^{n-1}} \frac{x_n}{|x-(z,0)|^n}  \dif z \right) \frac{\dif x}{x_n^n} \\
=& C \|f\|_{L^1(\mathbb{H}^n)}.
\end{align*}
Similarly,
\begin{align*}
\int_{\mathbb{R}^{n-1}} \int_0^{\infty}  \int_{S(z,r)} \left( |\nabla_g \phi|_g^n + |\phi|_g^n \right) \, \dif V'_g \frac{\dif r}{r^{n-1}} \dif z = C (\|\nabla_g \phi\|_{L^n(\mathbb{H}^n)}^n + \| \phi\|_{L^n(\mathbb{H}^n)}^n).
\end{align*}
Altogether,  (\ref{eq:avHntemp2}) (and hence \eqref{eq:avHn}) is bounded by
$$
C\|f\|_{L^1(\mathbb{H}^n)} (\|\nabla_g \phi\|_{L^n(\mathbb{H}^n)} + \| \phi\|_{L^n(\mathbb{H}^n)}).
$$
This is almost what we want, except that on the right hand side we have an additional $\|\phi\|_{L^n(\mathbb{H}^n)}$. To fix this, one applies Lemma~\ref{lem:SobHn} below, with $p = n$, and the desired conclusion of Theorem~\ref{thm:hyp} follows.
\end{proof}

\begin{lem} \label{lem:SobHn}
Assume $n \geq 2$. For any compactly supported smooth vector field $\phi$ on $\mathbb{H}^{n}$, and any $1 \leq p < \infty$, we have
$$
\|\phi\|_{L^{p}(\mathbb{H}^{n})} \leq C \|\nabla_g \phi\|_{L^{p}(\mathbb{H}^{n})}.
$$
\end{lem}

\begin{proof}
In fact, for any function $\Phi \in C^{\infty}_c(\mathbb{H}^{n})$, and any exponent $1 \leq p < \infty$, we have, from Hardy's inequality, that
\begin{equation} \label{eq:HardyLp}
\|\Phi\|_{L^p(\mathbb{H}^{n})} \leq C \left\|e_n \Phi \right\|_{L^p(\mathbb{H}^{n})}.
\end{equation}
This is because
$$
\int_0^{\infty} |\Phi(x)|^p \frac{\dif x_n}{x_n^{n}} \leq \left(\frac{p}{n-1}\right)^p \int_0^{\infty} |e_n \Phi|^p (x) \frac{\dif x_n}{x_n^{n}}
$$
by Hardy's inequality. \eqref{eq:HardyLp} then follows by integrating over all $x' \in \mathbb{R}^{n-1}$ with respect to $\dif x'$. Now we apply \eqref{eq:HardyLp} to $\Phi = \left\langle \phi, e_j \right\rangle$, $1 \leq j \leq n-1$. In view of \eqref{eq:covar_nj}, we have
$$
e_n \Phi  = \langle \nabla_{e_n} \phi, e_j \rangle + \langle \phi, \nabla_{e_n} e_j \rangle = \langle \nabla_{e_n} \phi, e_j \rangle,
$$
so we get
$$
\|\langle \phi, e_j \rangle \|_{L^p(\mathbb{H}^{n})} \leq C \|\nabla_g \phi\|_{L^p(\mathbb{H}^{n})}.
$$
Similarly, we can apply \eqref{eq:HardyLp} to $\Phi = \langle \phi, e_n \rangle$, and use \eqref{eq:covar_nn} in place of \eqref{eq:covar_nj}. Altogether we see that
$$
\|\phi\|_{L^{p}(\mathbb{H}^{n})} \leq C \|\nabla_g \phi\|_{L^{p}(\mathbb{H}^{n})},
$$
as desired.
\end{proof}

We now turn to the proofs of Lemma~\ref{lem:partsHn} and~\ref{lem:decompHn}.

\begin{proof}[Proof of Lemma~\ref{lem:partsHn}]
Note that $\langle f, \nu \rangle_g \psi = \langle \psi f, \nu \rangle_g$, and $\nu$ is the inward unit normal to $\partial X$. Also $\dif V'_g$ agrees with the induced surface measure on $S$ from $\mathbb{H}^n$. So by the divergence theorem on $\mathbb{H}^n$, we have
$$
\int_S \langle f, \nu \rangle_g \psi \, \dif V'_g = - \int_X \operatorname{div}_g\, (\psi f) \dif V_g.
$$
But  since $\operatorname{div}_g\, f = 0$, we have
$$\operatorname{div}_g (\psi f) = \langle f, \nabla_g \psi \rangle_g + \psi \operatorname{div}_g f = \langle f, \nabla_g \psi \rangle_g ,$$ and the desired equality follows.
\end{proof}

The proof of Lemma~\ref{lem:decompHn} will be easy, once we establish the following lemma:

\begin{lem} \label{lem:decompHmsimple}
Let $\varphi$ be a smooth function with compact support on $\mathbb{H}^m$, $m \geq 1$. For any $p > m$ and $\lambda > 0$, there exists a decomposition
$$
\varphi = \varphi_1 + \varphi_2 \quad \text{on $\mathbb{H}^m$},
$$
such that $\varphi_2$ is smooth with compact support on $\mathbb{H}^m$, and
$$
\|\varphi_1\|_{L^{\infty}(\mathbb{H}^m)} \leq C \lambda^{1-\frac{m}{p}} \|\nabla_g \varphi\|_{L^p(\mathbb{H}^m)},
$$
with
$$
\|\nabla_g \varphi_2\|_{L^{\infty}(\mathbb{H}^m)} \leq C \lambda^{-\frac{m}{p}} \|\nabla_g \varphi\|_{L^{p}(\mathbb{H}^m)}.
$$
\end{lem}

We postpone the proof of this lemma to the end of this section.

\begin{proof}[Proof of Lemma~\ref{lem:decompHn}]
Suppose $\varphi$ and $\lambda$ are as in Lemma~\ref{lem:decompHn}. We identify $S$ with $\mathbb{H}^m$, where $m=n-1$. (This is possible because the restriction of the metric of $\mathbb{H}^n$ to $S$ induces a metric on $S$ that is isometric to that of $\mathbb{H}^m$.) Using Lemma~\ref{lem:decompHmsimple}, with $p = n$, we obtain a decomposition $\varphi = \varphi_1 + \varphi_2$ on $S$, such that
$$
\|\varphi_1\|_{L^{\infty}(S)} \leq C \lambda^{\frac{1}{n}} \|\nabla_g \varphi\|_{L^{n}(S)},
$$
with
\begin{equation} \label{eq:varphi2bdd}
\|\nabla_g \varphi_2\|_{L^{\infty}(S)} \leq C \lambda^{\frac{1}{n}-1} \|\nabla_g \varphi\|_{L^{n}(S)}.
\end{equation}
We then extend $\varphi_2$ to $\mathbb{H}^n$ by setting for \((x_1, x'', x_n) \in \mathbb{R} \times \mathbb{R}^{n - 2} \times \mathbb{R}_+\)
\[
 \Tilde{\varphi}_2 (x_1, x'', x_n) = \varphi_2 \bigl(0, x'', \sqrt{x_1^2 + x_n^2}\bigr).
\]
One immediately checks that $\tilde{\varphi}_2$ is smooth with compact support on $\mathbb{H}^n$, with
\[
 \Vert \nabla_g \Tilde{\varphi}_2\Vert_{L^\infty (\mathbb{H}^n)} 
 \le \Vert \nabla_g \varphi_2 \Vert_{L^\infty (S)}.
\]
In view of (\ref{eq:varphi2bdd}), we obtain the desired bound for $\nabla_g \tilde{\varphi}_2$.
\end{proof}

It remains to prove Lemma~\ref{lem:decompHmsimple}.

\begin{proof}[Proof of Lemma~\ref{lem:decompHmsimple}]
When $m = 1$, the 1-dimensional hyperbolic space $\mathbb{H}^1$ is isometric to $\mathbb{R}$, and Lemma~\ref{lem:decompHmsimple} follows from its counterpart on $\mathbb{R}$ (see e.g. \cite{MR2078071}). Alternatively, it will follow from our treatment in the case $m \geq 2$, $0 < \lambda < 1$ below. 

So assume from now on $m \geq 2$. Suppose $\varphi$ is smooth with compact support on $\mathbb{H}^m$, $p > m$ and $\lambda > 0$. We will construct our desired decomposition $\varphi = \varphi_1 + \varphi_2$.
Recall that since \(p > m 
\), the Morrey inequality on \(\mathbb{H}^m\) implies that 
\begin{equation} \label{eq:MorreyHmLinfty}
  \Vert \varphi \Vert_{L^\infty (\mathbb{H}^m)} \le C \Vert \nabla_g \varphi \Vert_{L^p (\mathbb{H}^m)}.
\end{equation}
To see this, let $\zeta \in C^{\infty}_c(\mathbb{R})$ is a cut-off function such that $\zeta(s) = 1$ if $|s| \leq 1/2$, and $\zeta(s) = 0$ if $|s| \geq 1$. Let $x_0 := (0,1) \in \mathbb{H}^m$, and let $\zeta_{x_0}(x) := \zeta(d(x,x_0))$, where $d$ is the hyperbolic distance on $\mathbb{H}^m$. Consider the localization $\zeta_{x_0} \varphi$ of $\varphi$, to the unit ball centered at $x_0$. It satisfies
\begin{equation} \label{eq:W1pnorm}
\| \zeta_{x_0} \varphi \|_{W^{1,p}(\mathbb{R}^m)} \leq C \|\nabla_g \varphi\|_{L^p(\mathbb{H}^m)},
\end{equation}
where the left-hand side is a shorthand for $$\|\zeta_{x_0} \varphi \|_{L^p(\mathbb{R}^m)} + \|\nabla_e (\zeta_{x_0} \varphi)\|_{L^p(\mathbb{R}^m)};$$ here $\nabla_e$ denotes the Euclidean gradient of a function. (\ref{eq:W1pnorm}) holds because by the support of $\zeta_{x_0}$, we have
$$
\| \zeta_{x_0} \varphi \|_{L^p(\mathbb{R}^m)} \simeq \| \zeta_{x_0} \varphi \|_{L^p(\mathbb{H}^m)} \leq \|\nabla_g \varphi\|_{L^p(\mathbb{H}^m)},
$$ 
and
\begin{align*}
\|\nabla_e (\zeta_{x_0} \varphi)\|_{L^p(\mathbb{R}^m)}
\leq & \|(\nabla_e \zeta_{x_0}) \varphi \|_{L^p(\mathbb{R}^m)} + \|\zeta_{x_0} (\nabla_e \varphi)\|_{L^p(\mathbb{R}^m)} \\
\leq & C (\|\varphi\|_{L^p(\mathbb{H}^m)} + \|\nabla_g \varphi\|_{L^p(\mathbb{H}^m)}) \\
\leq &  C \|\nabla_g \varphi\|_{L^p(\mathbb{H}^m)}
\end{align*}
where we have used Lemma~\ref{lem:SobHn} in the last inequalities (note that Lemma~\ref{lem:SobHn} applies now since $m \geq 2$). In particular, by Morrey's inequality on $\mathbb{R}^m$, we get, from (\ref{eq:W1pnorm}), that 
$$
|\zeta_{x_0}(x) \varphi(x) - \zeta_{x_0}(y) \varphi(y)| \leq C \|\nabla_g \varphi\|_{L^p(\mathbb{H}^m)} |x-y|^{1-\frac{m}{p}}
$$
for all $x,y \in \mathbb{R}^m$. Taking $x = x_0$ and $y \in \mathbb{H}^m$ such that $d(y,x_0) = 2$, we get that 
$$
|\varphi(x_0)| \leq C \|\nabla_g \varphi\|_{L^p(\mathbb{H}^m)}.
$$
Since the isometry group of $\mathbb{H}^m$ acts transitively on $\mathbb{H}^m$, and since the right hand side of the above inequality is invariant under isometries, we obtain
$$
|\varphi(x)| \leq C \|\nabla_g \varphi\|_{L^p(\mathbb{H}^m)}.
$$
for all $x \in \mathbb{H}^m$, and hence (\ref{eq:MorreyHmLinfty}).

In particular, in view of (\ref{eq:MorreyHmLinfty}), when \(\lambda \ge 1\), it suffices to take \(\varphi_1 = \varphi\) and \(\varphi_2 = 0\). We then get the desired estimates for $\varphi_1$ and $\varphi_2$ trivially.

On the other hand, suppose now $0 < \lambda < 1$. We fix a compactly supported smooth function $\eta \in C^\infty_c (\mathbb{R}^m)$, with $$\int_{\mathbb{R}^m} \eta(v) \dif v = 1.$$
For $x= (x', x_m) \in \mathbb{H}^m$, we define
$$
\varphi_2(x) = \int_{\mathbb{R}^{m}} \varphi(x'+x_m e^{v_m} v', x_m e^{v_m}) \lambda^{-m} \eta(\lambda^{-1} v ) \dif v
$$
(where we wrote $v = (v',v_m) \in \mathbb{R}^{m-1} \times \mathbb{R}$), and define
$$
\varphi_1(x) = \varphi(x) - \varphi_2(x).
$$
Note that $\varphi_2$ is smooth with compact support on $\mathbb{H}^m$, and hence so is $\varphi_1$. Now for $i = 1, 2, \dots, m-1$, we have
$$
(e_i \varphi_2)(x,y) = \int_{\mathbb{R}^{m}} e^{-v_m} (e_i \varphi)(x'+x_m u' e^{v_m}, x_m e^{v_m}) \lambda^{-m} \eta(\lambda^{-1}v) \dif v.
$$
Since $v \mapsto \eta(\lambda^{-1} v)$ has support uniformly bounded with respect to $0 < \lambda < 1$, we have $e^{-v_m} \leq C$ on the support of the integral, where $C$ is independent of $0 < \lambda < 1$. 
Hence by H\"older's inequality, we have
\begin{align*}
&\|e_i \varphi_2\|_{L^{\infty}(\mathbb{H}^m)}\\
 \leq & C \left( \int_{\mathbb{R}^{m}} |e_i \varphi|^p(x'+x_m e^{v_m} v', x_m e^{v_m}) \dif v \right)^{\frac{1}{p}} \left\| \lambda^{-m} \eta(\lambda^{-1}v) \right\|_{L^{p'}(\dif v)}\\
= &  C \lambda^{-\frac{m}{p}} \left( \int_{\mathbb{R}^{m}_+} |e_i \varphi|^p (z) \frac{\dif z }{z_m^{m}} \right)^{\frac{1}{p}},
\end{align*}
the last line following from the changes of variables 
$z_m = e^{v_m}$, and then $z' = x' + z_m v'$.
We thus see that
$$
\|e_i \varphi_2\|_{L^{\infty}(\mathbb{H}^m)} \leq C \lambda^{-\frac{m}{p}} \|e_i \varphi\|_{L^p(\mathbb{H}^m)},
$$
as desired.

Furthermore, when $i = m$,
\begin{multline*}
(e_m \varphi_2)(x)  
= \int_{\mathbb{R}^{m}} \Bigl[(e_m \varphi)(x'+x_m v'e^{v_m}, x_m e^{v_m})\\
+ \sum_{i=1}^{m-1} v_i (e_i\varphi) (x'+x_m v'e^{v_m}, x_m e^{v_m})\Bigr]  \lambda^{-m} \eta( \lambda^{-1}v) \dif v.
\end{multline*}
Using that $|v_i| \leq C$ on the support of the integrals (uniformly in $0 < \lambda < 1$), and H\"older's inequality as above, we see that
$$
\|e_m \varphi_2\|_{L^{\infty}(\mathbb{H}^m)} \leq C \lambda^{-\frac{m}{p}} \|\nabla_g \varphi\|_{L^p(\mathbb{H}^m)},
$$
as desired.

Finally, to estimate $\varphi_1$, note that
$$
\varphi(x) = \lim_{\lambda \to 0^+}  \int_{\mathbb{R}^{m}} \varphi(x' +x_m e^{v_m} v', x_m e^{v_m}) \lambda^{-m} \eta(\lambda^{-1}v) \dif v.
$$
Hence
\begin{equation*}
\begin{split}
\varphi_1(x) &= \varphi(x) - \varphi_2(x) \\
&= -\int_0^{\lambda} \int_{\mathbb{R}^m} \varphi(x'+x_m e^{v_m} v', x_m e^{v_m}) \frac{d}{ds} \left[ s^{-m} \eta(s^{-1}v) \right] \dif v  \dif s.
\end{split}
\end{equation*}
But
\begin{equation*}
-\frac{d}{ds} \left[ s^{-m} \eta( s^{-1}v) \right] 
= \sum_{i=1}^{m} \frac{\partial}{\partial v_i} \left[ s^{-m} (v_i\eta)  (s^{-1}v) \right]
\end{equation*}
so we can plug this back in the equation for $\varphi_1$, and integrate by parts in $v$. Now
$$
\frac{\partial}{\partial v_i} [\varphi(x'+x_m e^{v_m} v', x_m e^{v_m})] = (e_i \varphi)(x'+x_m e^{v_m} v' ,x_m e^{v_m}), \quad i = 1, \dots, m - 1.
$$
and 
\begin{multline*}
 \frac{\partial}{\partial v_m} [\varphi(x'+x_m e^{v_m} v', x_m e^{v_m})]
 =(e_m \varphi)(x'+x_m e^{v_m} v' ,x_m e^{v_m})\\
 + \sum_{i = 1}^{m - 1} v_i  (e_i \varphi)(x'+x_m e^{v_m} v' ,x_m e^{v_m})
\end{multline*}
Hence
\begin{align*}
&\varphi_1(x) \\
=& -\int_0^{\lambda}  \sum_{i=1}^{m} \int_{\mathbb{R}^m} (e_i \varphi)(x' + x_m e^{v_m} v',x_m e^{v_m}) s^{-m} (v_i \eta)(s^{-1}v) \dif v \dif s \\
&-\int_0^{\lambda}\sum_{i=1}^{m - 1} \int_{\mathbb{R}^m} v_i (e_i \varphi)(x' + x_m e^{v_m} v',x_m e^{v_m}) s^{-m} (v_i \eta)(s^{-1}v) \dif v \dif s .
\end{align*}
When $0 < \lambda < 1$, the integral in $v$ in each term can now be estimated by H\"older's inequality, yielding
\begin{align*}
\|\varphi_1\|_{L^{\infty}(\mathbb{H}^n)}
\leq & \int_0^{\lambda} C s^{-\frac{m}{p}} \|\nabla_g \varphi\|_{L^p(\mathbb{H}^m)} \dif s \\
= & C \lambda^{1-\frac{m}{p}} \|\nabla_g \varphi\|_{L^p(\mathbb{H}^m)}.
\end{align*}
This completes the proof of Lemma~\ref{lem:decompHmsimple}.
\end{proof}

\appendix

\section{Indirect proof of Theorem~\ref{thm:hyp}}
\label{appendix}

We will give an alternative proof of Theorem~\ref{thm:hyp} from the following variant of Theorem~\ref{thm:vs}, whose proof can be found, for instance, in Van Schaftingen \cite{MR2078071} (it can also be deduced by a small modification of the proof we gave above of Theorem~\ref{thm:vs}):

\begin{prop}[Van Schaftingen \cite{MR2078071}] \label{prop:vs}
Suppose $f$ is a smooth vector field on $\mathbb{R}^n$ (not necessarily $\operatorname{div}\, f = 0$). Then for any compactly supported smooth vector field $\phi$ on $\mathbb{R}^n$, we have
\begin{equation} \label{eq:mainest_loword}
\left| \int_{\mathbb{R}^n} \langle f , \phi \rangle \right| \leq C \left( \|f\|_{L^1} \|\nabla \phi\|_{L^n} + \|\operatorname{div} f\|_{L^1} \|\phi\|_{L^n} \right),
\end{equation}
where $\langle \cdot, \cdot \rangle$ is the pointwise Euclidean inner product of two vector fields in $\mathbb{R}^n$.
\end{prop}

To prove Theorem~\ref{thm:hyp}, we consider a function \(\zeta \in C^\infty_c (\mathbb{R})\)
and we define for \(\alpha \in \mathbb{H}^n\) the function \(\zeta_\alpha : \mathbb{H}^n \to \mathbb{R} \) by
\[
   \zeta_\alpha (x) = \zeta \bigl(d (x, \alpha)\bigr).
\]
We assume that 
\[
 \int_{\mathbb{H}^n} \zeta_\alpha^2 (x)\dif V_g (\alpha) = 1.
\]
Now given vector fields $f$ and $\phi$ as in Theorem~\ref{thm:hyp}, we write
$$
\int_{\mathbb{H}^n} \langle f,\phi \rangle_g \dif V_g = \int_{\mathbb{H}^n}  \int_{\mathbb{H}^n} \langle \zeta_{\alpha} f(x), \zeta_{\alpha} \phi(x) \rangle_g \dif V_g (z) \dif V_g (\alpha).
$$
If $\alpha = (0,1) \in \mathbb{H}^n$, then
$$
\int_{\mathbb{H}^n} \langle \zeta_{\alpha} f(x), \zeta_{\alpha} \phi(x) \rangle_g \dif V_g
= \int_{\mathbb{R}^n_+}
\langle \zeta_{\alpha} f, \frac{\zeta_{\alpha} \phi}{x_n^{n+2}} \rangle_e\dif x ,
$$
where $\langle \cdot,\cdot \rangle_e$ is the Euclidean inner product of two vectors. Hence by Proposition~\ref{prop:vs}, this last integral is bounded by
$$
C \left( \|\zeta_{\alpha} f\|_{L^1(\mathbb{R}^n)} \left\| \nabla_e \left( \frac{\zeta_{\alpha} \phi}{x_n^{n + 2}} \right) \right\|_{L^n(\mathbb{R}^n)} + \|\langle \nabla_e(\zeta_{\alpha}), f \rangle_e\|_{L^1(\mathbb{R}^n)} \left\|  \frac{\zeta_{\alpha} \phi}{x_n^{n + 2}} \right\|_{L^n(\mathbb{R}^n)} \right)
$$
where we write $\nabla_e$ to emphasize that the gradients are with respect to the Euclidean metric. Now on the support of $\zeta_{\alpha}$, we have $|x_n| \simeq 1$, so altogether we get
\begin{align}  \label{eq:oneballHn}
&\left| \int_{\mathbb{H}^n} \langle \zeta_{\alpha} f, \zeta_{\alpha} \phi \rangle_g \dif V_g  \right| \\
\leq &
C \left( \|\zeta_{\alpha} f\|_{L^1(\mathbb{H}^n)} + \|\langle \nabla_g(\zeta_{\alpha}), f \rangle_g\|_{L^1(\mathbb{H}^n)} \right)
 \left( \| \nabla_g \phi \|_{L^n(\mathbb{H}^n)} +  \| \phi \|_{L^n(\mathbb{H}^n)} \right). \notag
\end{align}
This remains true even if $\alpha \ne (0,1)$, since there is an isometry mapping \(\alpha\) to $(0,1)$, and since \eqref{eq:oneballHn} is invariant under isometries of $\mathbb{H}^n$. By integrating with respect to $\alpha \in \mathbb{H}^n$,  we see that
$$ 
\left| \int_{\mathbb{H}^n} \langle  f,  \phi \rangle_g \dif V_g \right|
\leq
C \|f\|_{L^1(\mathbb{H}^n)}  \left( \| \nabla_g \phi \|_{L^n(\mathbb{H}^n)} +  \| \phi \|_{L^n(\mathbb{H}^n)} \right).
$$ 
We now use Lemma \ref{lem:SobHn} to bound $\|\phi\|_{L^n(\mathbb{H}^n)}$ by $\| \nabla_g \phi \|_{L^n(\mathbb{H}^n)}$. This concludes our alternative proof of Theorem~\ref{thm:hyp}.

\end{document}